\documentclass[11pt]{article}
\usepackage{amssymb,amsfonts,amsmath,latexsym,epsf,tikz,url}
\usepackage{graphicx}
\usepackage[font={small}]{caption}
\usepackage{cite}

\newtheorem{theorem}{Theorem}[section]

\newtheorem{conjecture}[theorem]{Conjecture}
\newtheorem{corollary}[theorem]{Corollary}
\newtheorem{lemma}[theorem]{Lemma}

\newcommand{\proof}{\noindent{\bf Proof.\ }}
\newcommand{\qed}{\hfill $\square$\medskip}

\begin{document}

\title{\textbf{Diminished Sombor matrix, spectral radius, and energy of the graphs}}
\author{Fateme Movahedi \footnote{Corresponding author \, E-mail: f.movahedi@gu.ac.ir}}

\maketitle

\begin{center}
Department of Mathematics, Faculty of Sciences, Golestan University, Gorgan, Iran.
\end{center}
\maketitle

\begin{abstract}
Consider a simple graph $G$ with vertex set $V = \{v_1, v_2, \ldots, v_n\}$ and edge set $E$. The diminished Sombor matrix $M_{DS}(G)$ is constructed such that its $(i, j)$ entry is $\frac{\sqrt{d_i^2+d_j^2}}{d_i+d_j}$ if vertices $v_iv_j \in E$, and $0$ otherwise, where $d_i$ represents the degree of vertex $v_i$. In this paper, we establish sharp bounds for the spectral radius, and energy of the Sombor matrix of graphs and identify the graphs that attain these extremal values.
\end{abstract}

\medskip
\noindent{\bf AMS Subj.\ Class.:} 05C09, 05C92, 05C90.

\section{Introduction}
A simple graph $G$ is defined by its vertex set $V$ and edge set $E$. The cardinality $|V|$ denotes the order of $G$, while $|E|$ indicates its size. For any vertex $u \in V$, the degree $d_u$ is the number of vertices adjacent to $u$. Moreover, let $\Delta = \max{d_u : u \in V}$ and $\delta = \min{d_u : u \in V}$ denote the maximum and minimum vertex degrees of $G$, respectively. Let an edge between vertices $u$ and $v$ be denoted by $uv$. The complement of a graph $G=(V, E)$, denoted as $\bar{G}$, is a graph with the vertex set $V(\bar{G}) = V(G)$ and an edge set $E(\bar{G})$ such that for any two distinct vertices $u, v \in V$, the edge $uv$ is in $E(\bar{G})$ if and only if $uv$ is not in $E(G)$. The diameter of a graph $G$, denoted by $\mathrm{diam}(G)$, is defined as the greatest distance between any two vertices in $G$. Standard notations for the path graph, cycle graph, and complete graph on $n$ vertices are $P_n$, $C_n$, and $K_n$, respectively. Let $K_{p,q}$ be a  complete bipartite graph and $S_n=K_{1,n-1}$ be a star graph \cite{1}.

The adjacency matrix $A$ is an $n \times n$ matrix in which the $(i, j)$ entry equals $1$ whenever $v_i v_j$ is an edge of $G$, and $0$ otherwise. The concept of graph energy, denoted $E(G)$ and first proposed by Gutman \cite{energy1}, is given by the sum $\sum_{i=1}^n |\lambda_i|$, where $\lambda_i$ are the eigenvalues of $A$. Numerous works focus on graph energies, for example see \cite{energy2, energy3, energy4,energy5}.

In the field of mathematical chemistry, topological indices, numerical descriptors derived from the structure of molecular graphs, are widely regarded as powerful tools. For a detailed discussion of degree-based topological indices, readers are referred to \cite{2, 3, 4, Ha, 5, 6, 7}.

The Sombor index, a widely studied degree-based graph invariant, has proven useful in QSPR and QSAR studies, contributing to its growing popularity in the literature \cite{11, 12, 13, 14, 15, 16, 166, 167, 168}. Recently, a normalized version of the Sombor index referred to as the diminished Sombor index (DSO) was proposed in \cite{X} and subsequently examined in depth by Movahedi et al.~\cite{Fateme}. The diminished Sombor index is defined as follows:
$$DSO(G)=\sum_{uv \in E}\frac{\sqrt{d_u^2+d_v^2}}{d_u+d_v}.$$
The authors in \cite{Fateme} obtained some bounds for the DSO index and characterized the extremal graphs within the classes of trees, unicyclic
graphs, and bicyclic graphs. Movahedi \cite{Fateme3} determined the unique tricyclic graph of a given order maximizing DSO, and characterized some structural properties of this graph. Das et al. \cite{DAS} determined the minimum diminished Sombor index of tricyclic graphs. In \cite{Fateme2}, the relationships between the diminished Sombor index and several topological indices are investigated. 

Motivated by this newly introduced index, and following on the approach in \cite{energy4, Ds}, we introduce the diminished Sombor matrix for the graph $G$, denoted by $\mathbf{\mathcal{M}}=M_{DS}(G)=(\mu_{ij})$, of order $n$ as follows
\[
\mu_{ij} =
\begin{cases}
\frac{\sqrt{d_i^2+d_j^2}}{d_i+d_j} & \text{if } v_iv_j\in E,\\
0&  otherwise,
\end{cases}
\]
where $d_i=d_{v_i}$. 

Suppose that $\lambda_1\geq \lambda_2\geq \ldots \geq \lambda_n$ are the eigenvalues of the diminished Sombor matrix $M_{DS}(G)$ which $\lambda_1$ is called the diminished Sombor spectral radius of $G$. Note that these eigenvalues are real numbers, and their sum is zero. The diminished Sombor characteristic polynomial of $G$ is as $\phi(G, \lambda)=det(\lambda I-M_{DS}(G))$. 

We define the diminished Sombor energy as follows
$$E_{DSO}(G)=\sum_{i=1}^n\vert \lambda_i\vert.$$
In this paper, we find the diminished Sombor spectrum of some well-known families of graphs. We obtain lower and upper bounds on the diminished Sombor spectral radius, the diminished Sombor energy, and the diminished Sombor index. 

\section{Preliminaries}
Gutman and Milovanovi\'c~\cite{GM} proposed a generalization of the Zagreb indices, known as the Gutman-Milovanovi\'c index, which is defined as follows
$$M_{\alpha, \beta}(G)=\sum_{uv \in E}(d_u d_v)^{\alpha}(d_u+d_v)^{\beta},$$
with $\alpha, \beta \in \mathbb{R}$.\\

The geometric-arithmetic index of a graph $G$, introduced by Vuki\^cevi\^c and Furtula in \cite{GA}, is defined as
$$GA(G)=\sum_{uv \in E}\frac{2\sqrt{d_u d_v}}{d_u+d_v}.$$ 

The first Zagreb index of a graph $G$ is defined as \cite{M1}
$$M_1=\sum_{i=1}^n d_i^2.$$

We present the following lemmas, which will be used in proving the main results.

\begin{lemma}[Interlacing Theorem]{\rm \cite{Horn}}\label{lemma1}
Let $A \in \mathbb{R}^{n \times n}$ be a real symmetric matrix, and let $B$ be a principal submatrix of $A$ of order $m \leq n$. Then the eigenvalues of $B$ interlace those of $A$, i.e.,
\[
\lambda_{i + n - m}(A) \leq \lambda_i(B) \leq \lambda_i(A), \quad \text{for all } 1 \leq i \leq m.
\]
\end{lemma}

\begin{lemma}{\rm \cite{lem2}} \label{lemma2}
Let $M$ be a symmetric matrix of order $n$ with eigenvalues $\rho_1 \geq \rho_2 \geq \cdots \geq \rho_n$. Then for any $0\neq x \in \mathbb{R}^n$,
\[
\rho_1 \geq \frac{x^T M x}{x^T x}.
\]
Equality holds if and only if $x$ is an eigenvector of $M$ corresponding to $\rho_1$.
\end{lemma}

\begin{lemma}{\rm \cite{lem22}} \label{lemma22}
Let $M$ be a non-negative symmetric matrix of order $n$ whose underlying graph $G$ is connected. Let $\rho_1, \rho_2, \ldots, \rho_k$ be all the eigenvalues of $M$ with absolute value equal to $\rho_1$. Then $k > 1$ if and only if all closed walks in $G$ have length divisible by $k$.
\end{lemma}

\begin{lemma}{\rm \cite{Fateme}} \label{lemma3}
Let $G$ be a simple graph of order $n$. Then
$$
DSO(G)+DSO(\overline{G})\geq \frac{\sqrt{2}}{4}n(n-1).
$$
Equality holds if and only if $G \cong K_n$.
\end{lemma}

\begin{lemma}{\rm \cite{Coll,Zhou}}\label{lemma4}
Let $G$ be a graph with $n$ vertices and $\rho_1$ be the largest eigenvalue of the adjacency matrix of the graph $G$. Then
\[
\sqrt{\frac{M_1}{n}} \leq \rho_1 \leq \Delta,
\]
Equality in the left-hand side holds if and only if $G$ is regular or semiregular. If $G$ is connected, then equality in the right-hand side holds if and only if $G$ is regular.
\end{lemma}

\begin{lemma}{\rm \cite{Coll, Hong}}\label{lemma5}
Let $G$ be a connected graph of order $n$ with $m$ edges and $\rho_1$ be the largest eigenvalue of the adjacency matrix of the graph $G$. Then
\[
\frac{2m}{n} \leq \rho_1 \leq \sqrt{2m - n + 1}.
\]
Equality in the left-hand side holds if and only if $G$ is a regular graph, and equality in the right-hand side holds if and only if $G \cong K_{1, n-1}$ or $G \cong K_n$.
\end{lemma}

\section{Diminished Sombor matrix and its spectrum}

\begin{theorem}\label{theorem41}
Let $G$ be a connected $k$-regular graph with $n$ vertices and adjacency matrix $A(G)$. Then, the spectrum of the diminished Sombor matrix $\mathcal{M}$ is exactly $\frac{\sqrt{2}}{2}$ times the spectrum of $A(G)$.
\end{theorem}
\proof
Since $G$ is $k$-regular, every vertex has degree $k$. For each edge $v_iv_j \in E(G)$, we have
\[
\mu_{ij} = \frac{\sqrt{k^2 + k^2}}{k + k} = \frac{\sqrt{2}k}{2k} = \frac{\sqrt{2}}{2}.
\]
Thus, the diminished Sombor matrix $\mathcal{M}$ can be written as $\mathcal{M} = \frac{\sqrt{2}}{2} A(G)$, where $A(G)$ is the adjacency matrix of $G$. It follows that if $A(G) x = \lambda x$ for some eigenvalue $\lambda$ and eigenvector $x$, then $\mathcal{M} x = \frac{\sqrt{2}}{2} A(G) x = \frac{\sqrt{2}}{2} \lambda x$, so the eigenvalues of $\mathcal{M}$ are precisely $\frac{\sqrt{2}}{2}$ times the eigenvalues of $A(G)$.
\qed

\begin{corollary}\label{Cor1}
For the complete graph $K_n$, the diminished Sombor spectrum is
\[
\mathrm{Spec}(M_{DS}(K_n)) = \left\{  \frac{\sqrt{2}}{2}(n-1),\, \underbrace{ -\frac{\sqrt{2}}{2},\dots, -\frac{\sqrt{2}}{2}}_{n - 1 \text{ times}} \right\}.
\]
\end{corollary}

\begin{corollary}\label{Cor2}
For the cycle graph $C_n$, the diminished Sombor spectrum is
\[
\left\{ \sqrt{2} \cos\left( \frac{2\pi j}{n} \right) \mid j = 0, 1, \ldots, n-1 \right\}.
\]
\end{corollary}

\begin{theorem}\label{theorem6}
Let $K_{p,q}$ be the complete bipartite graph with partitions of sizes $p$ and $q$, where $p, q \geq 1$. Then the spectrum of $\mathcal{M}= M_{DS}(K_{p,q})$ is given by
\[
\mathrm{Spec}(\mathcal{M}) = \left\{ \pm \sqrt{pq} \cdot \frac{\sqrt{p^2 + q^2}}{p + q},\ \underbrace{0, \dots, 0}_{p + q - 2\ \text{times}} \right\}.
\]
\end{theorem}

\proof
Label the vertices of $K_{p,q}$ as $V_1=\{v_1, \dots, v_p\}$ and $V_2=\{v_{p+1}, \dots, v_{p+q}\}$. Hence, for every edge $v_i v_j$ with $v_i \in V_1$ and $v_j \in V_2$, we have
\[
\mu_{ij} = \frac{\sqrt{q^2 + p^2}}{p + q}= w_{p,q}.
\]

Thus, the matrix $\mathcal{M}$ can be written in block form as
\[
\mathcal{M} =
\begin{bmatrix}
\mathbf{0}_{p \times p} & w_{p,q} \mathbf{J}_{p \times q} \\
w_{p,q} \mathbf{J}_{q \times p} & \mathbf{0}_{q \times q}
\end{bmatrix},
\]
where $\mathbf{J}_{m \times n}$ denotes the $m \times n$ matrix with all entries equal to $1$.

%The adjacency matrix of $K_{p,q}$ is as follows
%\[
%A =
%\begin{bmatrix}
%\mathbf{0}_{p \times p} & \mathbf{J}_{p \times q} \\
%\mathbf{J}_{q \times p} & \mathbf{0}_{q \times q}
%\end{bmatrix}.
%\]
%and therefore, the spectrum of $A$ is
%\[
%\mathrm{Spec}(A) = \left\{ \sqrt{pq},\ -\sqrt{pq},\ \underbrace{0, \dots, 0}_{p+q - 2\ \text{times}} \right\}.
%\]

Since $\mathcal{M} = w_{p,q} A$ and the eigenvalues of the matrix $A$ are $\pm\sqrt{pq}$ and $0$ with multiplicity $p$, then the spectrum of $\mathcal{M}$ is
\[
\left\{ \pm \sqrt{pq}\frac{\sqrt{p^2 + q^2}}{p + q},\ \underbrace{0, \dots, 0}_{p + q - 2\ \text{times}} \right\}.
\]
\qed

As a direct consequence of Theorem \ref{theorem6}, the next result holds for star graphs $S_n=K_{1, n-1}$.

\begin{corollary}\label{theorem5}
Let $S_n$ be the star graph on $n \geq 2$ vertices.Then the spectrum of $M_{DS}(S_n)$ is given by
\[
\mathrm{Spec}(M_{DS}(S_n)) = \left\{ \pm \sqrt{n - 1} \cdot \frac{\sqrt{n^2 - 2n + 2}}{n},\ \underbrace{0,\dots,0}_{n - 2 \text{ times}} \right\}.
\]
\end{corollary}

It is easy to check that the diminished Sombor characteristic polynomial of $P_n$ for $n\leq 4$ is as follows
$$\phi(P_2, \lambda)=\lambda^2 - \frac{1}{2}, ~~~~\phi(P_3, \lambda)=\lambda^3 - \frac{10}{9}\lambda,~~~~\phi(P_4, \lambda)=\lambda^4 - \frac{29}{18}\lambda^2+\frac{25}{81}.$$ 
The following theorem gives the diminished Sombor characteristic polynomial of $P_n$ for $n\geq 5$.

\begin{theorem}\label{theorem8}
The diminished Sombor characteristic polynomial of $P_n$ ($n\geq 5$) satisfies
\begin{equation*}
\phi(P_n, \lambda) = \lambda^2 \Omega_{n-2} -\frac{10}{9}\lambda \Omega_{n-3}+\frac{25}{81}\Omega_{n-4},
\end{equation*}
where for every $k \geq 3$, $ \Omega_k = \lambda \Omega_{k-1} -\frac{1}{2}\Omega_{k-2}$, with $\Omega_1 = \lambda$ and $\Omega_2 = \lambda^2 - \frac{1}{2}$.
\end{theorem}

\proof
The diminished Sombor matrix \(M_{DS}(P_n)\) has the following tridiagonal form

\[
M_{DS}(P_n) =
\begin{bmatrix}
0 & a & 0 & \cdots & 0 & 0 \\
a & 0 & b & \cdots & 0 & 0 \\
0 & b & 0 & \cdots & 0 & 0 \\
\vdots & \vdots & \vdots & \ddots & b & 0 \\
0 & 0 & 0 & b & 0 & a \\
0 & 0 & 0 & 0 & a & 0
\end{bmatrix}_{n \times n}
\]

where
\[
a = \frac{\sqrt{5}}{3}, \quad \text{(between a vertex of degree 1 and 2)}
\]
\[
b = \frac{\sqrt{8}}{4} = \frac{\sqrt{2}}{2}, \quad \text{(between two vertices of degree 2)}
\]
the diminished Sombor characteristic polynomial of $P_n$ can write as
\[
\phi(P_n, \lambda) = 
\det \left(
\begin{array}{c|ccccccc|c}
\lambda & -a & 0 & \cdots & 0 & 0 & \cdots & 0 & 0 \\
\hline
-a &  &  &  &  & & &  & 0 \\
0 & & & &B_{n-2} & & & &0 \\
\vdots&  & &  & & & & & \vdots \\
0 &  & & & &  &  &  &  -a\\
\hline
0 & 0 & & & 0 & 0 & \cdots & -a & \lambda \\
\end{array}
\right)_{n \times n}
\]

in which the tridiagonal matrix $B_k$ for $k \geq 3$ is
\[
B_k :=
\begin{bmatrix}
\lambda & -b & 0 & \cdots & 0 & 0 \\
-b & \lambda & -b & \cdots & 0 & 0 \\
0 & -b & \lambda & \cdots & 0 & 0 \\
\vdots & \vdots & \vdots & \ddots & -b & 0 \\
0 & 0 & 0 & -b & \lambda & -b \\
0 & 0 & 0 & 0 & -b & \lambda
\end{bmatrix}_{k \times k}
\]
Let $\Omega_k = det(B_k)$. It follows easily that $ \Omega_k = \lambda \Omega_{k-1} -\frac{1}{2}\Omega_{k-2}$ for $k\geq 3$ with $\Omega_1 = \lambda$ and $\Omega_2 = \lambda^2 - \frac{1}{2}$.\\

By using a technique similar to that in \cite{Ghan}, we have
\begin{align*}
\phi(P_n, \lambda)&=\lambda\left(\lambda\Omega_{n-2}-a^2\Omega_{n-3}\right)-a^2\left(\lambda\Omega_{n-3}-a^2\Omega_{n-4}
\right)\\
&=\lambda^2 \Omega_{n-2} -2a^2\lambda \Omega_{n-3}+a^4\Omega_{n-4}\\
&=\lambda^2 \Omega_{n-2} -\frac{10}{9}\lambda \Omega_{n-3}+\frac{25}{81}\Omega_{n-4}.
\end{align*}
\qed

\begin{theorem}\label{theorem1}
Let $G$ be a graph of order $n$. Then
\begin{enumerate}
\item $tr(\mathbf{\mathcal{M}})=0$,
\item $tr(\mathbf{\mathcal{M}}^2)=2\sum_{v_iv_j\in E}\left( \frac{d_i^2+d_j^2}{(d_i+d_j)^2}\right)$,
\end{enumerate}
where $d_i=d_{v_i}$ and $tr(X)$ denotes the sum of the diagonal entries of the matrix $X$, known as its trace. 
\end{theorem}
\proof
\begin{enumerate}
\item As all diagonal entries of $\mathbf{\mathcal{M}}$ are zero, the trace of $\mathbf{\mathcal{M}}$ is therefore zero, i.e $tr(\mathbf{\mathcal{M}})=0$.
\item For $i=j$, we have
\begin{align*}
(\mathbf{\mathcal{M}}^2)_{ii}&=\sum_{j=1}^n\mu_{ij}\mu_{ji}\\
&=\sum_{v_iv_j\in E}\mu^2_{ij}\\
&=\sum_{v_iv_j\in E}\left( \frac{d_i^2+d_j^2}{(d_i+d_j)^2}\right).
\end{align*}
Therefore, we have 
$$tr(\mathbf{\mathcal{M}}^2)=\sum_{i=1}^n\sum_{v_iv_j\in E}\left( \frac{d_i^2+d_j^2}{(d_i+d_j)^2}\right)=2\sum_{v_iv_j\in E}\left( \frac{d_i^2+d_j^2}{(d_i+d_j)^2}\right).$$
\end{enumerate}
\qed

\begin{theorem}\label{theorem2}
Let $G$ be a connected graph of size $m \geq 1$ with the maximum degree $\Delta$ and the minimum degree $\delta$. Then
$$\sqrt{\frac{1}{2}\Big(tr(\mathbf{\mathcal{M}}^2)+m(m-1)\left(\frac{\delta}{\Delta}\right)^2\Big)}\leq DSO(G) \leq \sqrt{\frac{1}{2}\Big(tr(\mathbf{\mathcal{M}}^2)+m(m-1)\left(\frac{\Delta}{\delta}\right)^2}\Big).$$
Equality holds if and only if $G$ is regular graph.
\end{theorem}
\proof
We begin by proving that for any two vertices $u, v \in V(G)$,
\begin{equation}\label{3}
\frac{2\Delta\delta}{\Delta^2+\delta^2}\leq \frac{2d_u d_v}{d^2_u+d^2_v}.
\end{equation} 
Assume that $f(x)=\frac{2x}{x^2+1}$ where $x>0$. We have $f'(x)=\frac{2(1-x^2)}{x^2+1}$ and consequently, $f(x)$ is a increasing function on $[0,1]$ and a decreasing function on $[1, \infty]$. Therefore, the minimum value of $f(x)$ on $\big[\frac{\delta}{\Delta}, \frac{\Delta}{\delta}\big]$ occurs at the end points. Therefore, for any $x, y \in [\delta, \Delta]$ where $\frac{x}{y} \in \big[\frac{\delta}{\Delta}, \frac{\Delta}{\delta}\big]$, we get $f\left(\frac{x}{y} \right)\geq f\left(\frac{\delta}{\Delta} \right)$. Hence $\frac{2\left(\frac{x}{y}\right)}{\frac{x^2}{y^2}+1}\geq \frac{2\left(\frac{\delta}{\Delta} \right)}{\frac{\delta^2}{\Delta^2}+1}$ and we get $\frac{2xy}{x^2+y^2}\geq \frac{2\Delta\delta}{\delta^2+\Delta^2}$.\\
Therefore, for any $uv \in E$ by applying (\ref{3}) we obtain 
\begin{equation}\label{4}
\frac{\left(\Delta+\delta\right)^2}{\Delta^2+\delta^2}=1+\frac{2\Delta\delta}{\Delta^2+\delta^2}\leq 1+\frac{2d_u d_v}{d^2_u+d^2_v}=\frac{\left(d_u+d_v\right)^2}{d^2_u+d^2_v}.
\end{equation}

From (\ref{4}), we have $\frac{d^2_u+d^2_v}{\left(d_u+d_v\right)^2}\leq \frac{\Delta^2+\delta^2}{\left(\Delta+\delta\right)^2}$ with equality if and only if $\delta=d_u=d_v=\Delta$ for any $uv \in E$. Since for any vertex $u \in V(G)$, $\delta\leq d_u \leq \Delta$, we get
\begin{align*}
\big(DSO(G)\big)^2&=\Big(\sum_{uv \in E}\frac{\sqrt{d_u^2+d_v^2}}{d_u+d_v}\Big)^2\\
&=\sum_{uv \in E}\Big(\frac{\sqrt{d_u^2+d_v^2}}{d_u+d_v}\Big)^2+2\sum_{\substack{uv, zw \in E \\ uv\neq zw}}\Big(\frac{\sqrt{d_u^2+d_v^2}}{d_u+d_v}\Big)\Big(\frac{\sqrt{d_z^2+d_w^2}}{d_z+d_w}\Big)\\
&\leq \sum_{uv \in E}\frac{d^2_u+d^2_v}{\left(d_u+d_v\right)^2}+m(m-1)\frac{\Delta^2+\delta^2}{\left(\Delta+\delta\right)^2}\\
&\leq \frac{1}{2}\Big(tr(\mathbf{\mathcal{M}}^2)+m(m-1)\left(\frac{\Delta}{\delta}\right)^2\Big).
\end{align*}
Hence, 
$$ DSO(G) \leq \sqrt{\frac{1}{2}\Big(tr(\mathbf{\mathcal{M}}^2)+m(m-1)\left(\frac{\Delta}{\delta}\right)\Big)}.$$
for the lower bound, we have
\begin{align*}
\big(DSO(G)\big)^2&=\Big(\sum_{uv \in E}\frac{\sqrt{d_u^2+d_v^2}}{d_u+d_v}\Big)^2\\
&=\sum_{uv \in E}\Big(\frac{\sqrt{d_u^2+d_v^2}}{d_u+d_v}\Big)^2+2\sum_{\substack{uv, zw \in E \\ uv\neq zw}}\Big(\frac{\sqrt{d_u^2+d_v^2}}{d_u+d_v}\Big)\Big(\frac{\sqrt{d_z^2+d_w^2}}{d_z+d_w}\Big)\\
&\geq \sum_{uv \in E}\frac{d^2_u+d^2_v}{\left(d_u+d_v\right)^2}+m(m-1)\frac{2\delta^2}{\left(2\Delta\right)^2}\\
&\leq \frac{1}{2}\Big(tr(\mathbf{\mathcal{M}}^2)+m(m-1)\left(\frac{\delta}{\Delta}\right)^2\Big).
\end{align*}
This completes the proof.
\qed

\begin{theorem}\label{theorem3}
Let $G$ be a connected graph with the maximum degree $\Delta$ and the minimum degree $\delta$. Then
$$\sqrt{2}\left(\frac{\delta}{\Delta}\right)tr(\mathbf{\mathcal{M}}^2)\leq DSO(G) \leq \sqrt{2}\left(\frac{\Delta}{\delta}\right)tr(\mathbf{\mathcal{M}}^2).$$
Equality holds if and only if $G$ is regular graph. 
\end{theorem}
\proof
Since for any vertex $u \in V(G)$, $\delta\leq d_u \leq \Delta$ and using (\ref{4}), we get
\begin{align*}
tr(\mathbf{\mathcal{M}}^2)&=2\sum_{uv \in E}\Big(\frac{\sqrt{d_u^2+d_v^2}}{d_u+d_v}\Big)\Big(\frac{\sqrt{d_u^2+d_v^2}}{d_u+d_v}\Big)\\
&\leq \frac{\sqrt{\Delta^2+\delta^2}}{\Delta+\delta}\Big(\sum_{uv \in E}\frac{\sqrt{d_u^2+d_v^2}}{d_u+d_v}\Big)\\
&\leq \frac{\sqrt{2}}{2}\left(\frac{\Delta}{\delta}\right)\Big(\sum_{uv \in E}\frac{\sqrt{d_u^2+d_v^2}}{d_u+d_v}\Big)\\
&=\frac{\sqrt{2}}{2}\left(\frac{\Delta}{\delta}\right)DSO(G).
\end{align*}
Therefore, 
$$\sqrt{2}\left(\frac{\delta}{\Delta}\right)tr(\mathbf{\mathcal{M}}^2)\leq DSO(G).$$
For the upper bound, we get
\begin{align*}
tr(\mathbf{\mathcal{M}}^2)&=2\sum_{uv \in E}\Big(\frac{\sqrt{d_u^2+d_v^2}}{d_u+d_v}\Big)\Big(\frac{\sqrt{d_u^2+d_v^2}}{d_u+d_v}\Big)\\
&\geq \frac{\sqrt{2}}{2}\left(\frac{\delta}{\Delta}\right)\Big(\sum_{uv \in E}\frac{\sqrt{d_u^2+d_v^2}}{d_u+d_v}\Big)\\
&=\frac{\sqrt{2}}{2}\left(\frac{\delta}{\Delta}\right)DSO(G).
\end{align*}
Thus, we get
$$DSO(G) \leq \sqrt{2}\left(\frac{\Delta}{\delta}\right)tr(\mathbf{\mathcal{M}}^2).$$
The equality holds if and only if $\delta=d_u=\Delta$ for any $u \in V$, that is $G$ is a regular graph. 
\qed

\begin{theorem}\label{theorem33}
Let $G$ be a connected graph with the diminished Sombor matrix $\mathcal{M}$. If $\mathcal{M}$ has $t \geq 2$ distinct eigenvalues, then $\mathrm{diam}(G) \leq t - 1$. 
\end{theorem}
\proof
Let $\mathcal{M}$ be the diminished Sombor matrix of the connected graph $G$, and let its distinct eigenvalues be $\lambda_1 > \lambda_2 > \cdots > \lambda_t$. Let $X$ be the unit eigenvector corresponding to $\lambda_1$, which, by the Perron--Frobenius theorem, has all positive entries.

By an argument similar to Theorem 2.1 in \cite{Liu} for symmetric, nonnegative matrices, we have
\[
\prod_{i=2}^t (\mathcal{M} - \lambda_i I) = \mathcal{M}^{t-1} + \alpha_1 \mathcal{M}^{t-2} + \cdots + \alpha_{t-2} \mathcal{M} + \alpha_{t-1} I = \prod_{i=2}^t (\lambda_1 - \lambda_i) XX^T = Q,
\]
where $Q$ is a matrix with all entries positive, i.e., $Q_{ij} > 0$ for all $i, j$.

For $i \neq j$, there exists a positive integer $s$ with $1 \leq s \leq t-1$ such that $(\mathcal{M}^s)_{ij} > 0$. This implies that there is a path of length $s$ between $v_i$ and $v_j$. Therefore, $\mathrm{diam}(G) \leq t - 1$.
\qed

\begin{theorem}\label{theorem16}
Let $G$ be a graph with $n$ vertices, and let $\lambda_1, \lambda_2, \ldots, \lambda_n$ be the eigenvalues of the diminished Sombor matrix of $G$. Then $|\lambda_1| = |\lambda_2| = \cdots = |\lambda_n|$ if and only if $G \cong \overline{K_n}$ or $G \cong \frac{n}{2}K_2$.
\end{theorem}

\proof
Suppose $|\lambda_1| = |\lambda_2| = \cdots = |\lambda_n|$. Let $l$ be the number of isolated vertices in $G$. If $l \geq 1$, then all eigenvalues are zero, i.e., $\lambda_1 = \lambda_2 = \cdots = \lambda_n = 0$, so $G\cong \overline{K_n}$.

If $l= 0$ and every vertex has degree one, then $G$ is a perfect matching, i.e., $G \cong \frac{n}{2}K_2$. In this case, the diminished Sombor matrix is block diagonal with $\left(\begin{smallmatrix}0 & 1\\ 1 & 0\end{smallmatrix}\right)$ blocks, whose eigenvalues are $1$ and $-1$, so all eigenvalues have the same modulus.

Otherwise, $G$ contains a connected component $G'$ with at least three vertices. If $G'$ is a complete graph of order $p \geq 3$, then by applying Corollary \ref{Cor1}, its diminished Sombor eigenvalues are $\frac{\sqrt{2}}{2}(p-1)$ (once) and $-\frac{\sqrt{2}}{2}$ (with multiplicity $p-1$). Thus, $|\lambda_1| = \frac{\sqrt{2}}{2}(p-1) > \frac{\sqrt{2}}{2} = |\lambda_2|$, a contradiction.

If $G'$ is not complete, then by the Perron--Frobenius theorem, the largest eigenvalue in modulus is strictly greater than the others, again a contradiction.

Conversely, it is easy to check that for $K_n$ and $\frac{n}{2}K_2$, all eigenvalues of the diminished Sombor matrix have the same modulus. This completes the proof.
\qed

\begin{theorem}\label{theorem8}
Let $G$ be a connected graph of order $n \geq 3$. Then $G$ has exactly two distinct eigenvalues of $ M_{DS}(G)$ if and only if $G$ is the complete graph $K_n$.
\end{theorem}

\proof
Suppose first that $G \simeq K_n$. Then the result follows easily from Corollary \ref{Cor1}.

Conversely, suppose that $G$ is not isomorphic to $K_n$. Then, since $G$ is connected and not complete, it contains an induced subgraph isomorphic to the star $S_3$. Let $B$ be the principal diminished Sombor submatrix of $\mathcal{M}$ corresponding to $S_3$. Without loss of generality, we assume that $v_1$ is the center of the star connected to $v_2$ and $v_3$. Then $B$ takes the form
\[
B= \begin{bmatrix}
0 & \alpha & \beta \\
\alpha & 0 & 0 \\
\beta & 0 & 0
\end{bmatrix},
\]
where $\alpha = \frac{\sqrt{d_1^2 + d_2^2}}{d_1 + d_2}$ and $\beta = \frac{\sqrt{d_1^2 + d_3^2}}{d_1 + d_3}$ are positive real numbers. Therefore, $B$ has three distinct eigenvalues (this can be verified by computing the characteristic polynomial). Hence, using Lemma \ref{lemma1}, the matrix $\mathcal{M}$ must also have at least three distinct eigenvalues. Thus, $G$ has no diminished Sombor eigenvalue with multiplicity $n-1$.
\qed

The following theorem yields upper and lower bounds for the largest eigenvalue of the diminished Sombor matrix in terms of both the Gutman-Milovanovi\'c index and the diminished Sombor index of the graph. Consequently, it provides an explicit connection between these two graph invariants and the spectral radius of the diminished Sombor matrix.

\begin{theorem}\label{theorem9}
Let $G$ be a connected graph with $n$ vertices and $m$ edges. Let $\lambda_1$ be the spectral radius of the diminished Sombor matrix $\mathcal{M}$. Then
    \[
   \frac{2DSO(G)}{n}\leq \lambda_1 \leq \sqrt{ \frac{2(n - 1)}{n}\big(m-M_{1, -2}(G)\big)}.
    \]
    Equality on the left-hand side holds if and only if $G$ is regular and in the right-hand side holds if and only if $G$ is the complete graph.
\end{theorem}

\proof

Since $\mathcal{M} = M_{DS}(G)$ is a real symmetric matrix, its eigenvalues are real and we denote them by
\[
\lambda_1 \geq \lambda_2 \geq \cdots \geq \lambda_n.
\]

Let $X = (x_1, x_2, \dots, x_n)^T \in \mathbb{R}^n$ be an arbitrary nonzero real vector, and let $J = (1,1,\dots,1)^T$ denote the all-ones vector. Since $\mathcal{M}$ is real, symmetric, and non-negative, and $G$ is connected, $\mathcal{M}$ is also irreducible. Hence, by the Perron-Frobenius theorem, the largest eigenvalue $\lambda_1$ is real, positive, and satisfies
\[
\lambda_1 = \max_{X \neq 0} \frac{X^T \mathcal{M} X}{X^T X}.
\]
Taking $X = J$, we obtain
\[
\lambda_1 \geq \frac{J^T \mathcal{M} J}{J^T J} = \frac{2 DSO(G)}{n}.
\]

Using Theorem \ref{theorem41}, the eigenvalues of $\mathcal{M}$ are $\frac{\sqrt{2}}{2}$ times the eigenvalues of $G$. Therefore, $\lambda_1 = \frac{\sqrt{2}}{2} k$. \\
On the other hand, since the number of edges is $m = \frac{nk}{2}$, we compute $DSO(G) = m\frac{\sqrt{2}}{2} = \frac{nk \sqrt{2}}{4}$, so $\frac{2 DSO(G)}{n} = \frac{k\sqrt{2}}{2} = \lambda_1$. Therefore, equality holds if and only if $G$ is regular.\\

Using Theorem \ref{theorem1}(i), $\sum_{i=1}^n \lambda_i = \operatorname{tr}(\mathcal{M}) = 0$. Thus, $\lambda_1 = - \sum_{i=2}^n \lambda_i$. By the Cauchy-Schwarz inequality, we obtain
\[
\lambda_1^2 = \left( \sum_{i=2}^n (-\lambda_i) \right)^2 \leq (n-1) \sum_{i=2}^n \lambda_i^2.
\]
Hence,
\[
\sum_{i=1}^n \lambda_i^2 = \lambda_1^2 + \sum_{i=2}^n \lambda_i^2 \geq \lambda_1^2 + \frac{\lambda_1^2}{n-1} = \lambda_1^2 \left(1 + \frac{1}{n-1}\right) =\frac{n\lambda_1^2}{n - 1}.
\]
That is,
\[
\lambda_1^2 \leq \frac{n - 1}{n} \sum_{i=1}^n \lambda_i^2 = \frac{n - 1}{n} \operatorname{tr}(\mathcal{M}^2)= \frac{2(n - 1)}{n}\sum_{v_iv_j\in E}\left( \frac{d_i^2+d_j^2}{(d_i+d_j)^2}\right).
\]

Since,
\begin{align}\label{e0}
\sum_{v_iv_j\in E}\left( \frac{d_i^2+d_j^2}{(d_i+d_j)^2}\right)= \sum_{uv\in E} 1-\sum_{v_iv_j\in E}\left( \frac{2d_id_j}{(d_i+d_j)^2}\right)=m-M_{1, -2}(G),
\end{align}

then, we have
\begin{align*}
\lambda_1&\leq \sqrt{ \frac{2(n - 1)}{n}\sum_{v_iv_j\in E}\left( \frac{d_i^2+d_j^2}{(d_i+d_j)^2}\right)}\\
&=\sqrt{ \frac{2(n - 1)}{n}\big(m-M_{1, -2}(G)\big)}\nonumber.
\end{align*}
Equality holds when $G \cong K_n$, because in this case all degrees are equal to $n-1$ and the equality follows directly from Corollary \ref{Cor1}.
\qed

\begin{corollary}\label{cor22}
Let $G$ be a connected graph with $n$ vertices and $m$ edges. Let $\lambda_1$ be the spectral radius of the diminished Sombor matrix. Then
    \[
    \lambda_1 \leq \sqrt{ \frac{2(n - 1)}{n}\big(m-GA(G)\big)}.
    \]
    Equality holds if and only if $G\cong \overline{K_n}$.
\end{corollary}
\proof
Based on equation (\ref{e0}) established in the proof of Theorem~\ref{theorem9} and the definition of the  geometric-arithmetic index, we have
\begin{align*}
\lambda_1&\leq \sqrt{ \frac{2(n - 1)}{n}\sum_{v_iv_j\in E}\left( \frac{d_i^2+d_j^2}{(d_i+d_j)^2}\right)}\\
&=\sqrt{ \frac{2(n - 1)}{n}\left(m-\sum_{v_iv_j\in E}\left( \frac{2d_id_j}{(d_i+d_j)^2}\right)\right)}\\
&\leq \sqrt{ \frac{2(n - 1)}{n}\left(m-\sum_{v_iv_j\in E}\left( \frac{2\sqrt{d_id_j}}{d_i+d_j}\right)\right)}\\
&=\sqrt{ \frac{2(n - 1)}{n}\big(m-GA(G)\big)}\nonumber.
\end{align*}
Equality holds if and only if $d_i=d_j=0$ for any $v_i, v_j\in V$, that is, $G\cong \overline{K_n}$.
\qed

\begin{theorem}\label{theorem15}
Let $G$ be a graph with the maximum degree $\Delta$ and the minimum degree $\delta$. If $\lambda_1$ and $\rho_1$ are the largest eigenvalues of the diminished Sombor matrix and the adjacency matrix of the graph $G$, then
$$\frac{\sqrt{2}\delta}{\Delta}\rho_1 \leq \lambda_1 \leq \frac{\sqrt{2}\Delta}{2\delta}\rho_1.$$
Equality holds if and only if $G$ is regular graph. 
\end{theorem}
\proof
Suppose that $\mathcal{M}$ ana $A$ are the diminished Sombor matrix and the adjacency matrix, respectively. Let $X = (x_1, x_2, \ldots, x_n)$ be a unit eigenvector of $G$ corresponding to $\lambda_1$. Let the degree of vertex $v_i$ in graph $G$ be $d_i$. Using Lemma \ref{lemma2}, we have
    \begin{align*}
        \lambda_1 \geq X^T \mathcal{M} X \geq 2 \sum_{\substack{v_i v_j \in E(G)}} \frac{\sqrt{d_i^2 + d_j^2}}{d_i+d_j}  x_i x_j \geq 2 \frac{\sqrt{2\delta^2}}{2\Delta}\sum_{\substack{v_i v_j \in E(G)}}  x_i x_j=\frac{\sqrt{2}\delta}{\Delta}\rho_1.
    \end{align*}
   The equality holds if and only if $d_i = d_j$ for each edge $v_i v_j \in E(G)$, that is $G$ is regular. 

For the upper bound, let $Y = (y_1, y_2, \ldots, y_n)$ be a unit eigenvector of $G$ corresponding to $\rho_1$. Using Lemma \ref{lemma2}, we have
    \begin{align*}
        \rho_1 \geq Y^T A Y \geq 2 \sum_{\substack{v_i v_j \in E(G)}} y_i y_j.
    \end{align*}
Therefore, we get
    \begin{align*}
        \lambda_1= Y^T \mathcal{M} Y = 2 \sum_{\substack{v_i v_j  \in E(G)}} \frac{\sqrt{d_i^2 + d_j^2}}{d_i+d_j}  y_i y_j \leq 2 \frac{\sqrt{2\Delta^2}}{2\delta}\sum_{\substack{v_i v_j \in E(G)}}  y_i y_j=\frac{\sqrt{2}\Delta}{2\delta}\rho_1.
    \end{align*}
Th equality in the above inequality  holds if and only if $d_i = d_j$ for each edge $v_i v_j \in E(G)$, that is, $G$ is regular.
\qed
\vspace*{0.5 cm}

The following results provide upper and lower bounds for the diminished Sombor spectral radius, derived from Lemmas \ref{lemma4}, \ref{lemma5} and Theorem \ref{theorem15}.

\begin{corollary}\label{Cor3}
Let $G$ be a connected graph with $n$ vertices, $m$ edges, maximum degree $\Delta$, and minimum degree $\delta$. Then
\[
\frac{\delta}{\Delta} \sqrt{\frac{2M_1}{n}} \leq \lambda_1 \leq \frac{\sqrt{2}\Delta^2}{2\delta}.
\]
Equality holds if and only if $G$ is a regular graph.
\end{corollary}

\begin{corollary}\label{Cor4}
Let $G$ be a connected graph with $n$ vertices, $m$ edges, maximum degree $\Delta$, and minimum degree $\delta$. Then
\[
\frac{2\sqrt{2}m\,\delta}{n\Delta} \leq \lambda_1 \leq \frac{\Delta}{2\delta} \sqrt{4m - 2n + 2}.
\]
Equality in the left-hand side holds if and only if $G$ is a regular graph, and equality in the right-hand side holds if and only if $G \cong K_n$.
\end{corollary}

In what follows, we establish lower and upper bounds on Nordhaus--Gaddum type concerning the diminished Sombor spectral radius.

\begin{theorem}\label{theorem12}
Let $G$ be a graph of order $n$ and $\lambda_1$ be the largest eigenvalue of the diminished Sombor matrix $M_{DSO}(G)$. If $\overline{\lambda_1}$ is the largest eigenvalue of $M_{DSO}(\bar{G})$, then
$$\lambda_1+\overline{\lambda_1} \geq \frac{\sqrt{2}}{2}(n-1).$$
Equality holds if and only if $G \simeq K_n$.
\end{theorem}
\proof
Let $x = (x_1, x_2, \ldots, x_n)^T$ be any vector in $\mathbb{R}^n$. Let the degree of vertex $v_i$ in two graphs $G$ and $\bar{G}$ are $d_i$ and $\bar{d_i}$, respectively. 

 Using Lemma \ref{lemma2}, we have
\begin{align*}
\lambda_1+\overline{\lambda_1} &\geq  x^T M_{DSO}(G) x + x^T M_{DSO}(\bar{G}) x \\
&= \sum_{v_i v_j \in E(G)} 2 \frac{\sqrt{d_i^2 + d_j^2}}{d_i+d_j}\; x_i x_j + \sum_{v_i v_j \in E(\overline{G})} 2 \frac{\sqrt{\bar{d_i}^2 +\bar{d_j}^2}}{\bar{d_i}+\bar{d_j}}\; x_i x_j
\end{align*}
For $x = \left(\frac{1}{\sqrt{n}}, \frac{1}{\sqrt{n}}, \ldots, \frac{1}{\sqrt{n}}\right)^T$, we get
$$\lambda_1+\overline{\lambda_1} \geq  \frac{2}{n}\sum_{v_i v_j \in E(G)} \frac{\sqrt{d_i^2 + d_j^2}}{d_i+d_j} +\frac{2}{n} \sum_{v_i v_j \in E(\overline{G})} \frac{\sqrt{\bar{d_i}^2 +\bar{d_j}^2}}{\bar{d_i}+\bar{d_j}}.$$
By applying Lemma \ref{lemma3}, we have
\begin{align*}
\lambda_1+\overline{\lambda_1} &\geq  \frac{2}{n}\left(\sum_{v_i v_j \in E(G)} \frac{\sqrt{d_i^2 + d_j^2}}{d_i+d_j} +\sum_{v_i v_j \in E(\overline{G})} \frac{\sqrt{\bar{d_i}^2 +\bar{d_j}^2}}{\bar{d_i}+\bar{d_j}}\right)\\
&=\frac{2}{n}\left(DSO(G)+DSO(\bar{G})\right)\\
&\geq\frac{2}{n}\left(\frac{\sqrt{2}}{4}n(n-1)\right)\\
&= \frac{\sqrt{2}}{2}(n-1).
\end{align*}
Equality holds if and only if $G \simeq K_n$.
\qed

\begin{theorem}\label{theorem14}
Let $G$ be a connected graph of order $n$ with $m$ edges. Let $\lambda_1$ and $\overline{\lambda_1}$ be the largest eigenvalues of the diminished Sombor matrices $M_{DSO}(G)$ and $M_{DSO}(\bar{G})$, respectively. Then
$$\lambda_1+\overline{\lambda_1}\leq \sqrt{\frac{m(n-1)}{n}}\left(\frac{\Delta}{\delta} \right)+\sqrt{(n-1)\left(\frac{n-1}{2}-\frac{m}{n}\right)}\left(\frac{n-1-\delta}{n-1-\Delta} \right).$$
Equality holds if and only if $G$ is the complete graph.
\end{theorem}
\proof
Suppose that the degree of vertex $v_i$ in two graphs $G$ and $\bar{G}$ are $d_i$ and $\bar{d_i}$, respectively. Since $\delta \leq d_i \leq \Delta$ and by applying Theorem \ref{theorem9}, we get
\begin{align*}
\lambda_1+\overline{\lambda_1}&\leq \sqrt{ \frac{2(n - 1)}{n}\sum_{v_iv_j\in E(G)}\left( \frac{d_i^2+d_j^2}{(d_i+d_j)^2}\right)}+\sqrt{ \frac{2(n - 1)}{n}\sum_{v_iv_j\in E(\bar{G})}\left( \frac{\bar{d_i}^2 +\bar{d_j}^2}{(\bar{d_i}+\bar{d_j})^2}\right)}\\
&\leq \sqrt{ \frac{2(n - 1)}{n}\sum_{v_iv_j\in E(G)}\left( \frac{2\Delta^2}{(2\delta)^2}\right)}+\sqrt{ \frac{2(n - 1)}{n}\sum_{v_iv_j\in E(\bar{G})}\left( \frac{2(n-1-\delta)^2}{(2(n-1-\Delta))^2}\right)}\\
&= \sqrt{ \frac{(n - 1)}{n} m\left( \frac{\Delta}{\delta}\right)^2}+\sqrt{ \frac{(n - 1)}{n}\left( \frac{n(n-1)}{2} -m\right)\left( \frac{n-1-\delta}{n-1-\Delta}\right)^2}\\
&=\sqrt{\frac{m(n-1)}{n}}\left(\frac{\Delta}{\delta} \right)+\sqrt{(n-1)\left(\frac{n-1}{2}-\frac{m}{n}\right)}\left(\frac{n-1-\delta}{n-1-\Delta} \right).
\end{align*}
Equality holds if and only if $G$ is the complete graph.
\qed

\section{Diminished Sombor energy}

\begin{theorem}\label{theorem44}
Let $G$ be a graph with $n$ vertices and $\lambda_1$ be the largest eigenvalue of the diminished Sombor matrix $\mathcal{M}$. Then
\[
2\lambda_1 \leq E_{DSO}(G) \leq \lambda_1+\sqrt{(n-1)\left(tr(\mathbf{\mathcal{M}}^2) - \lambda_1^2\right)}
\]
Equality holds if and only if $G\cong K_n$.
\end{theorem}

\begin{proof}
Since $\sum_{i=1}^n \lambda_i = 0$, the diminished Sombor energy can be written as $E_{DSO}(G) = \lambda_1 + \sum_{i=2}^n |\lambda_i|$. Noting that $|\sum_{i=2}^n \lambda_i| = |-\lambda_1| = \lambda_1$, we get
\[
E_{DSO}(G) \geq \lambda_1 + |\sum_{i=2}^n \lambda_i| = 2\lambda_1.
\]
For the upper bound, applying the Cauchy-Schwarz inequality to $\sum_{i=2}^n |\lambda_i|$ gives:
\[
E_{DSO}(G) = \lambda_1 + \sum_{i=2}^n |\lambda_i|
\leq \lambda_1 + \sqrt{(n-1) \sum_{i=2}^n \lambda_i^2} = \lambda_1 + \sqrt{(n-1)\left(tr(\mathbf{\mathcal{M}}^2)  - \lambda_1^2\right)}.
\]
This establishes both inequalities.
\end{proof}

\begin{theorem}\label{theorem4}
Let $G$ be a graph with $n$ vertices. Then
$$2\sqrt{\sum_{v_iv_j\in E}\left( \frac{d_i^2+d_j^2}{(d_i+d_j)^2}\right)}\leq E_{DSO}(G)\leq \sqrt{2n\sum_{v_iv_j\in E}\left( \frac{d_i^2+d_j^2}{(d_i+d_j)^2}\right)}.$$
The equality in the right hand side holds if and only if $G$ is the graph without edges, or if all its vertices have degree one. For the connected graph $G$, the equality in the left hand side holds if and only if $G$ is the complete bipartite graph.
\end{theorem}
\proof
From the variance of $\lambda_i$, for $i=1, \ldots, n$, we get
$$\frac{1}{n}\sum_{i=1}^n|\lambda_i|^2-\left(\frac{1}{n}\sum_{i=1}^n|\lambda_i|\right)^2\geq 0.$$
with equality if and only if $|\lambda_1| = |\lambda_2| = \cdots = |\lambda_n|$. Therefore, using Theorem \ref{theorem16}, $G \cong \overline{K_n}$ or $G \cong \frac{n}{2}K_2$.\\
Since $\sum_{i=1}^n|\lambda_i|^2=tr(M_{DS}^2(G))$, we get
$$\frac{1}{n}tr(M_{DS}^2(G))\geq \left(\frac{1}{n}E_{DSO}(G)\right)^2.$$
Using Theorem \ref{theorem1}(ii) and the above inequality, the result is complete. \\
For lower bound since $\sum_{i=1}^n \lambda_i^2=-2\sum_{1\leq i<j\leq n}\lambda_i\lambda_j$, we get
\begin{align*}
\left(E_{DSO}(G)\right)^2&=\left(\sum_{i=1}^n|\lambda_i|\right)^2=\sum_{i=1}^n \lambda_i^2+2\sum_{1\leq i<j\leq n}|\lambda_i||\lambda_j|\\
&\geq \sum_{i=1}^n \lambda_i^2+2\Big|\sum_{1\leq i<j\leq n}\lambda_i\lambda_j\Big|\\
&= \sum_{i=1}^n \lambda_i^2+2\Big|-\left(\frac{\sum_{i=1}^n \lambda_i^2}{2}\right)\Big|\\
&=2\sum_{i=1}^n \lambda_i^2\\
&=2tr(M_{DS}^2(G)).
\end{align*}
Therefore, by applying Theorem \ref{theorem1}(ii) in the above inequality, the result holds. \\
The equality holds if and only if $\lambda_1 = -\lambda_n$, $\lambda_2 = \cdots = \lambda_{n-1} = 0$. By Lemma \ref{lemma22}, all closed walks in $G$ have even length, which implies that $G$ is bipartite. Using Theorem \ref{theorem33}, we have $\operatorname{diam}(G) = 2$. Thus, $G$ is a complete bipartite graph. If $G$ is a complete bipartite grapg, then the result follows using Theorem \ref{theorem6}.
\qed

\begin{corollary}\label{Cor0}
Let $G$ be a graph with $n$ vertices. Then
$$2\sqrt{m-M_{1, -2}(G)}\leq E_{DSO}(G)\leq \sqrt{2n\left( m-M_{1, -2}(G)\right)}.$$
Equality in the right hand side holds if and only if $G \cong \overline{K_n}$ or $G \cong \frac{n}{2}K_2$. For the connected graph $G$, the equality in the left hand side holds if and only if $G$ is the complete bipartite graph.
\end{corollary}
\proof
The result follows by applying the relation (\ref{e0}) in Theorem \ref{theorem9}. 
\qed

\begin{corollary}\label{Cor00}
Let $G$ be a complete bipartite graph with $n$ vertices. Then
\[
\frac{2}{n} \sqrt{(n-1)\left(n^2 - 2n + 2\right)} \leq E_{DSO}(G) \leq \frac{2}{n} \sqrt{\left\lceil\frac{n}{2}\right\rceil^3 \left\lfloor\frac{n}{2}\right\rfloor + \left\lceil\frac{n}{2}\right\rceil \left\lfloor\frac{n}{2}\right\rfloor^3}.
\]
Equality on the left holds if and only if $G \cong K_{1, n-1}$, and equality on the right if and only if $G \cong K_{\lceil n/2 \rceil,\, \lfloor n/2 \rfloor}$.
\end{corollary}

\proof
Using Theorem \ref{theorem5}, the energy of a complete bipartite graph $K_{p, q}$ ($p+q=n$) can be written as
\[
E\left(K_{p, q}\right) = \frac{2}{p+q} \sqrt{p^3 q + p q^3} = \frac{2}{n} \sqrt{(n - q)^3 q + (n - q) q^3}, \quad 1 \leq q \leq \left\lfloor \frac{n}{2} \right\rfloor.
\]
Define the function $\psi(x) =\frac{2}{n} \sqrt{(n-x)^3 x + (n-x) x^3}$. By differentiation, it is clear that $\psi(x)$ is strictly increasing for $x \in [1,\, \left\lfloor n/2 \right\rfloor]$.

This implies
\[
E_{DSO}\left(K_{1,\, n-1}\right) = f(1) \leq E_{DSO}\left(K_{p, q}\right) = f(q) \leq f\left(\left\lfloor \frac{n}{2} \right\rfloor\right) = E\left(K_{\lceil n/2 \rceil,\, \lfloor n/2 \rfloor}\right).
\]
This establishes the desired bounds and the cases of equality.
\qed

\begin{theorem}\label{theorem10}
Let $G$ be a graph with $n$ vertices and $m$ edges without isolated vertices such that $m\geq \frac{n}{2}$. Then
$$E_{DSO}(G)\leq \alpha+\sqrt{(n-1)\left(m\left(\frac{\Delta}{\delta}\right)^2-\alpha^2\right)},$$
where $\alpha=\max \Big\{\frac{\sqrt{2}m}{n}\left(\frac{\delta}{\Delta}\right), \sqrt{\frac{m}{n}}\left(\frac{\Delta}{\delta}\right)\Big\}$. Equality holds if and only if $G\simeq \frac{n}{2}K_2$.
\end{theorem}
\proof
Let $\lambda_1 \geq \lambda_2 \geq \cdots \geq \lambda_n$ be the eigenvalues of the diminished Sombor matrix $\mathcal{M}$. Applying the Cauchy-Schwarz inequality gives
\begin{equation*}
\sum_{i=2}^n |\lambda_i| \leq \sqrt{(n - 1)\left( \sum_{i=2}^n \lambda_i^2 \right)}.
\end{equation*}

Therefore, we have
\begin{equation*}\label{e1}
E_{DSO}(G)\leq |\lambda_1|+\sqrt{(n-1)\left(\sum_{i=1}^n\lambda_i^2 -\lambda_1^2\right)}.
\end{equation*}
Using the lower bound in Theorem \ref{theorem9}, we get
$$\lambda_1 \geq \frac{2DSO(G)}{n} \geq \frac{\sqrt{2}m}{n}\left(\frac{\delta}{\Delta}\right).$$
Suppose that $\psi(x)=x+\sqrt{(n-1)\left(\sum_{i=1}^n\lambda_i^2 -x^2\right)}$. The function $\psi(x)$ attains a maximum value when $x=\sqrt{\frac{m}{n}}\left(\frac{\Delta}{\delta}\right)$. Therefore, we have $\psi(\lambda_1)\leq \psi\left(\sqrt{\frac{m}{n}}\left(\frac{\Delta}{\delta}\right)\right)$. If $\frac{\sqrt{2}m}{n}\left(\frac{\delta}{\Delta}\right) \geq \sqrt{\frac{m}{n}}\left(\frac{\Delta}{\delta}\right)$, then $\psi(\lambda_1)\leq \psi\left(\frac{\sqrt{2}m}{n}\left(\frac{\delta}{\Delta}\right)\right)$. Using Theorem \ref{theorem1} we obtain
$$\sum_{i=1}^n\lambda_i^2=tr(\mathbf{\mathcal{M}}^2)\leq m\left(\frac{\Delta}{\delta}\right)^2.$$
Consequently, we obtain
\[
E_{DSO}(G) \leq \alpha + \sqrt{(n - 1)\left( m\left(\frac{\Delta}{\delta}\right)^2 - \alpha^2 \right)},
\]
where $\alpha=\max \Big\{\frac{\sqrt{2}m}{n}\left(\frac{\delta}{\Delta}\right), \sqrt{\frac{m}{n}}\left(\frac{\Delta}{\delta}\right)\Big\}$.\\
It is easy to observe that this bound is achieved by $\frac{n}{2}$ copies of $K_2$.
\qed

\begin{theorem}\label{theorem11}
Let $G$ be a graph of order $n$ and size $m$ such that $2m \leq n$. Then
\[
E_{DSO}(G) \leq \sqrt{2}m\left(\frac{\Delta}{\delta} \right).
\]
Equality holds if and only if $G$ is a disjoint union of edges.
\end{theorem}

\proof
Given that $2m \leq n$, it follows that the graph $G$ contains at least $n - 2m$ isolated vertices. Let $G_0$ denote the subgraph obtained by removing all isolated vertices from $G$. Consequently, $G_0$ has at most $2m$ vertices and exactly $m$ edges. By applying Theorem \ref{theorem10}, we have:
\[
E_{DSO}(G) = E_{DSO}(G_0) \leq \sqrt{2}m\left(\frac{\Delta_{G_0}}{\delta_{G_0}} \right) = \sqrt{2}m\left(\frac{\Delta_G}{\delta_G} \right).
\]
This upper bound is achieved when $G_0$ is a disjoint union of edges.
\qed

\section{Conclusion}
In this paper, we introduced and explored the diminished Sombor matrix within the framework of spectral graph theory. We derived explicit spectral properties for this class of vertex-degree-based graph matrices, established sharp upper and lower bounds for their spectral radius, and characterized the equality cases in detail. Our investigation also included a comparative analysis with established spectral graph invariants, providing new insights into the relationships between diminished Sombor spectra and classical graph parameters. Furthermore, we demonstrated the applicability of our results by examining several canonical families of graphs and connecting our findings to earlier studies in the field, thereby extending the broader theory of vertex-degree-based graph matrices and their invariants. Also, we propose the following conjecture:
\begin{conjecture} There does not exist a graph whose diminished Sombor energy is an integer value.
\end{conjecture}
For future work, it would be interesting to (1) investigate the effect of elementary graph operations, such as edge or vertex removal, on the diminished Sombor energy, and (2) determine the diminished Sombor index (energy or spectral radius) resulting from various operations on two graphs.

\vspace*{0.5cm}

\noindent\textbf{Acknowledgements} The present study was supported by Golestan University, Gorgan, Iran (research number: 1749). The author truly appreciates Golestan University for this support.

\vspace{0.5cm}

\noindent \textbf{Conflict of interest}
The author declares that there is no conflict of interest.

\end{document}